\documentclass{birkjour}
\usepackage[noadjust]{cite}
\usepackage{xcolor}
\RequirePackage[all]{xy}
\usepackage{hyperref}

\newtheorem{thm}{Theorem}[section]
\newtheorem{cor}[thm]{Corollary}
\newtheorem{lem}[thm]{Lemma}

\theoremstyle{definition}

\theoremstyle{remark}
\newtheorem{rem}[thm]{Remark}

\numberwithin{equation}{section}

\def\cl#1#2{{{{C}\!{\ell}}_{{#1},{#2}}}}
\def\dl#1#2{{{{K}}_{{#1},{#2}}}}
\def\conj#1#2{{{#1}^{(#2)}}}
\def\proj#1#2{{\langle{#1}\rangle}_{#2}}
\def\dlsub#1#2#3{{{{K}}^{#3}_{{#1},{#2}}}}
\def\Tr{{\rm Tr}}
\def\tr{{\rm tr}}

\def\elt#1#2#3{#3 \in \dl{#1}{#2}}

\def\firstn#1{\{1,2,\hdots,{#1}\}}
\def\firstnwithzero#1{\{0,1,2,\hdots,{#1}\}}

\def\Det#1{{\rm Det}(#1)}
\def\det#1{{\rm det}(#1)}
\def\End#1{\mathrm{End}(#1)}
\def\Mat{{\rm Mat}}

\def\R{{\mathbb R}}
\def\C{{\mathbb C}}

\def\N{{\mathbb N}}
\def\Z{\mathbb{Z}}

\begin{document}

%
%
%
%
%
%
%
%

\title[Inverse in Commutative Analogues of Clifford Algebras]{Determinant, Characteristic Polynomial,\\ and Inverse in Commutative Analogues\\ of Clifford Algebras}

\begin{abstract}
Commutative analogues of Clifford algebras are algebras defined in the same way as Clifford algebras except that their generators commute with each other, in contrast to Clifford algebras in which the generators anticommute. In this paper, we solve the problem of finding multiplicative inverses in commutative analogues of Clifford algebras by introducing a matrix representation for these algebras and the notion of determinant in them. We give a criteria for checking if an element has a multiplicative inverse or not and, for the first time, explicit formulas for multiplicative inverses in the case of arbitrary dimension. The new theorems involve only operations of conjugation and do not involve matrix operations. We also consider notions of trace and other characteristic polynomial coefficients and give explicit formulas for them without using matrix representations.

\end{abstract}

\author[H. Sharma]{Heerak Sharma}
\address{
Indian Institute of Science Education and Research (IISER), Pune\\
Pune\\
India}

\email{heerak.sharma@students.iiserpune.ac.in}

\author[D. Shirokov]{Dmitry Shirokov}
\address{
HSE University\\
Moscow 101000\\
Russia
\medskip}
\address{
and
\medskip}
\address{
Institute for Information Transmission Problems of the Russian Academy of Sciences \\
Moscow 127051 \\
Russia}
\email{dm.shirokov@gmail.com}


\subjclass{Primary 15A15; Secondary 15A66}
\keywords{Characteristic polynomial, Clifford algebra, Commutative analogue, Determinant, Inverse, Tensor product, Trace.}
\date{\today}
\dedicatory{In memory of Professor Yuri Grigoriev}

\label{page:firstblob}
\maketitle

\section{Introduction}
The generators of a Clifford algebra anticommute. This anticommutativity is essential and allows the algebra to describe the plethora of things it is able to describe. It is then a natural question whether an analogous algebra whose generators commute with each other interesting or not. We call such an algebra a commutative analogue of Clifford algebra. We give their precise definition in Section \ref{Section: Definition of K_{p,q}s}. These algebras have previously been studied in different contexts. The most prominent examples of theses algebras are the multicomplex spaces and commutative quaternions.

Multicomplex spaces were first introduced by Segre in 1892 in his pioneering paper \cite{Segre}. They have been studied extensively and are full of beautiful results \cite{Price, MulticomplexIdeals, HolomorphyInMC, MulticomplexHyperfunctions, DifferentialEQnInMC}. The most well-known example of multicomplex spaces are the bicomplex numbers which have been studied extensively \cite{Shapiro, BCHolFunction, LunaBCAndElFun,FuncAnalBC,SingularititesBCHolFunc, GeometryIdentitityTheorems, FractalTypeSets, CauchyTypeIntegral, SingulFuncBCVar, BCHolFunCal, BicomplexHyperfunctions}. Commutative quaternions are isomorphic to bicomplex numbers and have been well studied. F. Catoni et al. in \cite{Catoni1, Catoni2} discussed commutative quaternions and in general some commutative hypercomplex numbers. H.H. Kösal and M. Tosun in \cite{CommQuatMat} considered matrices with entries from commutative quaternions.

Commutative analogues of Clifford algebras have also been considered before in other contexts. N. G. Marchuk in \cite{MarchukDemo, MarchukClassification} required the need for commutative analogues of Clifford algebras when he was considering tensor product decomposition of Clifford algebras and classifying tensor products of Clifford algebras. G. S. Staples, in his books \cite{StaplesBook1,StaplesBook2}, introduces symmetric Clifford algebras, which are commutative analogues of Clifford algebras. He uses them, along with zeons, to address combinatorial problems.

We came across these commutative analogues of Clifford algebras while studying multiplicative inverses in Clifford algebras. We found that subalgebras of Clifford algebras isomorphic to a commutative analogue of Clifford algebra are notorious in the sense that they force the formula for the determinant to be a linear combination instead of a single product. For an example, check out the subalgebra ${\rm{Span}_\R}\{1,e_{1256},e_{1346},e_{2345}\} \subseteq \cl{6}{0}$ described in \cite{A.Acus}. Naturally, we wanted to study more about these algebras hoping that they might aid us with finding multiplicative inverses in Clifford algebras. Also, as mentioned by Staples in his book \cite{StaplesBook2} (check Proposition 2.2), one can realise commutative analogues of Clifford algebras as subalgebras of some Clifford algebras and therefore studying commutative analogues of Clifford algebras could aid in study of Clifford algebras.

Coming back to the problem of finding multiplicative inverses, the problem of finding multiplicative inverses in multicomplex spaces has been considered before. To the best of our knowledge, no one has presented explicit formulas for a general multicomplex space before. For bicomplex numbers, the results are well known \cite{Price, BCHolFunction, Shapiro}. For tricomplex numbers, \cite{PlatonicSolids, DynamicsTricomplex} mention the result.  Price in his book \cite{Price} considers this problem and uses a representation of multicomplex spaces, that he calls Cauchy--Riemann matrices, to give criterias for invertiblity of elements and proves many results for small $n$. We use the same representation to give explicit formulas for general $n$ but we motivate this representation using the tensor product decomposition of commutative analogues of Clifford algebras discussed in \cite{CommAnalCliffAlg}. In our experience, our viewpoint on Cauchy--Riemann matrices would simplify many results given in Chapter 5 of Price's book and provide proofs of conjectures made at the end of the chapter.

The study, in this paper, is inspired from \cite{OnComputing}. We will introduce a matrix representation for commutative analogues of Clifford algebras, notions of trace, determinant, and other characteristic polynomial coefficients in them, and use operations of conjugations previously introduced in \cite{CommAnalCliffAlg} to find explicit formulas for them. Theorems \ref{thm: formula for trace}, \ref{thm: another formula for trace},  \ref{thm formula for determinant}, \ref{thm formula for characteristic polynomial coefficients},  \ref{formula for inverse theorem} and \ref{thm: non-invertible elements are zero divisors} are new. This paper is organized as follows: in Section \ref{Section: Definition of K_{p,q}s} we define the commutative analogues of Clifford algebras, introduce the notations that we will use in this paper and recall results about tensor product decomposition and operations of conjugation from \cite{CommAnalCliffAlg}. In Section \ref{Section: matrix rep}, we give a matrix representation for these algebras giving new perspective to Price's Cauchy--Riemann matrices. In Section \ref{Section: Notions of determinant, trace, characteristic polynomial}, we use this matrix representation to define notions of trace, determinant and characteristic polynomial. In Section \ref{Explicit formulas}, we use the operations of conjugations to give explicit formulas for trace, determinants, and characteristic polynomial coefficients. These new formulas do not involve matrix operations and involve only operations of conjugation in commutative analogues of Clifford algebras. In Section \ref{Section: mult inverses}, we discuss the multiplicative inverse problem. The conclusions follow in Section~\ref{section_conclusions}.

\section{Definitions, notations, tensor product decomposition and operations of conjugation}\label{Section: Definition of K_{p,q}s} 
\subsection{Definitions and notations}
Let $p,q \in \Z_{\ge0}$ and $n := p+q$. The \textit{commutative analogue of Clifford algebra}, denoted by $\dl{p}{q}$, is a real associative algebra with unity generated by the generators $\{e_1, e_2, \hdots, e_n\}$ which obey the following multiplication:
\begin{align*}
&{e_i}^2 = 1; \;\; \text{if} \;\; {1\le i \le p},\\
&{e_i}^2 = -1; \;\; \text{if} \;\; {p+1 \le i \le p+q :=n},\\
&{e_i}{e_j} = {e_j}{e_i} \;\; \forall \;\; i,j \in \firstn{n}.
\end{align*}
The multiplication is then extended to all of $\dl{p}{q}$ via bi-linearity. 

\subsubsection{Grades and projections}
Let $l\in\N$. For $1 \le i_1 < i_2 < \cdots < i_l \le n$ define $e_{i_1 i_2 \cdots i_l} := e_{i_1}e_{i_2} \cdots e_{i_l}$. Also, define $e_{\{\}} := 1$ and $e_{\{i_1, i_2, \cdots, i_l\}} := e_{i_1 i_2 \cdots i_l}$. $\dl{p}{q}$ as a vector space is the $N := 2^n$ dimensional space
$$\mathrm{Span}_\R\{e_A \;|\; A \subseteq \firstn{n} \}.$$
The elements $\{e_A \;|\; A \subseteq \firstn{n} \}$ form a basis for $\dl{p}{q}$.

A general element in $\dl{p}{q}$ is of the form:
\begin{multline*}
    U = u + \sum_{1 \le i \le n} u_ie_i + \!\!\sum_{1 \le i_1 < i_2 \le n}\!\! u_{i_1 i_2}e_{{i_1} {i_2}} + \quad\cdots\quad + \!\!\!\!\!\!\!\!\!\!\\ \sum_{1 \le i_i < i_2 < \cdots < i_{n-1} \le n}\!\!\!\!\!\!\!\!\!\! u_{i_1 i_2 \cdots i_{n-1}}e_{i_1 i_2 \cdots i_{n-1}} + u_{12\cdots n}e_{12\cdots n}.
\end{multline*}

Let $A\subseteq\firstn{n}$. By $|A|$ we denote the number of elements in $A$. For $k \in \firstnwithzero{n}$, we define \textit{subspace of grade $k$}, denoted by $\dlsub{p}{q}{k}$, as 
$$\dlsub{p}{q}{k} := {\rm{Span}}_{\R}\{e_A \;|\; A \subseteq \firstn{n} \text{ and } |A| = k\}$$

i.e., $\dlsub{p}{q}{k}$ is $\R$-span of basis elements of length k.

We define \textit{projection maps} $\proj{.}{k}: \dl{p}{q} \to \dlsub{p}{q}{k}$; $U \mapsto \proj{U}{k}$ which are linear maps which return the grade $k$ part of $U$. In general, 
$$U = \sum_{k=0}^{n}{\proj{U}{k}}.$$

\subsubsection{Examples}
\begin{enumerate} 
    \item One can check that multicomplex space with $n$ imaginary units is exactly $\dl{0}{n}$. For more details, we refer to reader to Section 3.3 of \cite{CommAnalCliffAlg}.

    \item The algebra $\dl{1}{1}$ is known as the algebra of commutative quaternions \cite{Catoni1, Catoni2, CommQuatMat}.

    \item The subalgebra $S = {\rm{Span}_\R}\{1,e_{1256},e_{1346},e_{2345}\} \subseteq \cl{6}{0}$ described in \cite{A.Acus} is isomorphic to $\dl{2}{0}$. This becomes apparant once we define $g_1 = e_{1256}, g_2 = e_{1346}$. Then $g_1 g_2 = g_2 g_1 = -e_{2345}$ and $g_1^2 = g_2^2 = +1$. Therefore, the subalgebra $S\cong\dl{2}{0}$ the isomorphism being identification of $g_1,g_2$ as generators of $\dl{2}{0}$. 
\end{enumerate}

\subsection{Tensor product decomposition}\label{Subsection: tensor product decomposition}
An described in \cite{CommAnalCliffAlg}, an immediate consequence of commutativity is that one can decompose any $\dl{p}{q}$ as tensor products of `$p$' $\dl{1}{0}$s and `$q$' $\dl{0}{1}$s.
\begin{thm}\label{tensor product decomposition theorem}
    It is the case that:
    \begin{equation}\label{tensor product decomposition result}
        \dl{p}{q} \cong \underbrace{\dl{1}{0} \otimes \dl{1}{0} \cdots \otimes \dl{1}{0}}_{p \; \text{times}} \otimes \underbrace{\dl{0}{1} \otimes \dl{0}{1} \cdots \otimes \dl{0}{1}}_{q \; \text{times}}.
    \end{equation}
\end{thm}
A useful result from the tensor product decomposition is that
\begin{equation}\label{Tensor decomposition Property 1}
\dl{p_1}{q_1}\otimes\dl{p_2}{q_2} \cong \dl{p}{q} \text{ when } p_1 + p_2 = p,\quad q_1 + q_2 = q.
\end{equation}

In context of multicomplex spaces, this decomposition captures the fact that an element of a multicomplex space $\mathbb{M}_n$ can be expressed as an element of $\mathbb{M}_k$ with coefficients from $\mathbb{M}_{n-k}$ instead of $\R$. 

\subsection{Operations of conjugation}\label{subsubsection: operations of conjugation}
We will need operations of conjugation in $\dl{p}{q}$ to give explicit expressions later on in this paper. One can think of these as a generalisation of operations of conjugation in complex numbers, split complex numbers. We defined operations of conjugation for $\dl{p}{q}$ in \cite{CommAnalCliffAlg}, and will recall the results that we need for this paper.

We define $n = p+q$ operations of conjugation in $\dl{p}{q}$ where each operation of conjugation negates one generator. More explicitly, for $l \in \firstn{n}$, we define $\conj{(.)}{l}: \dl{p}{q} \to \dl{p}{q}$ such that $\conj{U}{l} = U|_{e_l \to -e_l}$. 

In context of multicomplex spaces and commutative quaternions, operations of conjugation have been discussed. In bicomplex spaces, the $3$ operations of conjugations are well known \cite{Price,BCHolFunction,Shapiro}. For tricomplex spaces, operations of conjugation have been discussed in \cite{PlatonicSolids}. In general for multicomplex spaces, the notion of operations of conjugation is fairly new and has been discussed in \cite{MulticomplexIdeals, MSThesis}. In commutative quaternions, the operations of conjugations have been discussed in \cite{Catoni1,CommQuatMat}. One can check that for multicomplex spaces, the operations of conjugations defined above are a generalisation of operations of conjugation in bicomplex numbers and tricomplex numbers. 

We have the following results, of which only 4 \eqref{properties of operations of conjugations 4} is non-trivial. We have given a proof for that in \cite{CommAnalCliffAlg}.
\begin{lem}
\begin{enumerate}
\item The operations of conjugations are involutions\footnotemark, i.e., 
    \begin{equation}\label{properties of operations of conjugations 1}
        \conj{(\conj{U}{l})}{l} = U \text{ for all } U \in \dl{p}{q},\; l \in \firstn{n}.
    \end{equation}
    
\item The operations of conjugation are linear, i.e., 
    \begin{equation}\label{properties of operations of conjugations 2}
        \conj{(aU+bV)}{l} =a\conj{U}{l} + b\conj{V}{l} \text{ for all } U,V \in \dl{p}{q}, a,b\in \R, l \in \firstn{n}.
    \end{equation}
    
\item The operations of conjugations commute with each other, i.e., 
    \begin{equation}\label{properties of operations of conjugations 3}
        \conj{(\conj{U}{l_1})}{l_2} = \conj{(\conj{U}{l_2})}{l_1} \text{ for all } l_1, l_2 \in \firstn{n}.
    \end{equation}

\item Operations of conjugation distribute over multiplication.
    \begin{equation}\label{properties of operations of conjugations 4}
        \conj{(UV)}{l} = \conj{U}{l}\conj{V}{l} \text{ for all } U,V \in \dl{p}{q}, l \in \firstn{n}.
    \end{equation}
\end{enumerate}
\end{lem}
\footnotetext{For a more detailed study on involutions in multicomplex spaces, check out \cite{InvolutionsMC}.}

Keeping in mind property 3 \eqref{properties of operations of conjugations 3}, we denote superposition of many operations of conjugations by just writing them next to each other, for example, $\conj{(\conj{(\conj{U}{1})}{2})}{3} = \conj{U}{1)(2)(3}$. We introduce the following notation for a superposition of operations of conjugation: let $A = \{i_1,i_2,\hdots, i_k\} \subseteq \firstn{n}$, then we define for $U \in \dl{p}{q}$
\begin{equation}\label{notation for superposition of operations of conjugation}
    \conj{U}{A} := \conj{U}{i_1)(i_2)\ldots(i_k}.
\end{equation}
We also define $\conj{U}{\{\}} := U$. Let $A\subseteq\firstn{n}$. We call $\conj{U}{A}$ a conjugate of $U$. 

\begin{rem}\label{operations of conjugations form an abelian group}
    Note that the set of all superpositions of operations of conjugates form an Abelian group (as mentioned in \cite{SlicesOfMandelbrodSet, MulticomplexIdeals} for the particular case of multicomplex spaces). The group operation in the group is superposition of operations of conjugation and the identity element of this group is the operation of conjugation $\conj{(.)}{\{\}}$.
\end{rem}

Also, we make an observation. The only operation of conjugation that can eliminate all terms involving the generator $e_k$ is $\conj{(.)}{k}$. We usually eliminate $e_k$ in an element $\elt{p}{q}{U}$ by multiplying U with the conjugate $\conj{U}{k}$ as presented in the following lemma.

\begin{lem}\label{properties of operations of conjugations 5}
Let $\elt{p}{q}{U}$. Then the product $U\conj{U}{k}$ doesn't have the generator $e_k$ in it.
\end{lem}
\begin{proof}
    Let $\elt{p}{q}{U}$. Then by \eqref{Tensor decomposition Property 1}, $\dl{p}{q} \cong \dl{p^\prime}{q^\prime} \otimes \dl{1}{0} \text{ or } \dl{p^\prime}{q^\prime} \otimes \dl{0}{1}$, according to if $e_k^2 = +1 \text{ or } -1$, where $p^\prime + q^\prime = n-1$. Therefore, we can write $U$ as $U = x + ye_k$ where $x,y \in \dl{p^\prime}{q^\prime}$.\\

    \noindent Now, the product $U\conj{U}{k} = (x + ye_k)(x - ye_k) = x^2 -e_k^2y^2$. Since $x,y\in\dl{p^\prime}{q^\prime}$ implies $x^2,y^2\in\dl{p^\prime}{q^\prime}$ and since $e_k^2 = \pm1$, $U\conj{U}{k}$ doesn't have the generator $e_k$. 
\end{proof}

Note that the other operations of conjugations cannot eliminate terms involving $e_k$.

\section{A matrix representation for $\dl{p}{q}$}\label{Section: matrix rep}
Now that we have the tensor product decomposition \eqref{tensor product decomposition result}, we use it to define a representation for $\dl{p}{q}$. To do that, first we need representations of $\dl{1}{0}$ and $\dl{0}{1}$.
\subsection{Representations of $\dl{1}{0}$ and $\dl{0}{1}$}\label{faithful reps of Dl(1,0) and DL(0,1)}
The map $1 \mapsto \begin{bmatrix}
    1 & 0\\
    0 & 1\\
\end{bmatrix}$, ${e_1} \mapsto \begin{bmatrix}
    0 & 1\\
    1 & 0\\
\end{bmatrix}$ gives a faithful representation of $\dl{1}{0}$ on $\rm{Mat}(2,\R)$. A general element $a + {a_1}{e_1}$ maps to the matrix $\begin{bmatrix}
    a & a_1\\
    a_1 & a\\
\end{bmatrix}$. We denote this representation by $\beta_{+}$.

The map $1 \mapsto \begin{bmatrix}
    1 & 0\\
    0 & 1\\
\end{bmatrix}$, ${e_1} \mapsto \begin{bmatrix}
    0 & -1\\
    1 & 0\\
\end{bmatrix}$ gives a faithful representation of $\dl{0}{1}$ on $\rm{Mat}(2,\R)$. A general element $a + {a_1}{e_1}$ maps to the matrix $\begin{bmatrix}
    a & -a_1\\
    a_1 & a\\
\end{bmatrix}$. We denote this representation by $\beta_{-}$.

\begin{rem}
Note that in this paper we consider only real representations of $\dl{p}{q}$, since we want to deal with a real determinant. For example, $\dl{0}{1} \simeq \C$  also has a complex representation $1 \mapsto 1, e_1 \mapsto i$. Another example is $\dl{0}{2}$, the bicomplex space, whose complex representation is discussed in Section 2.4 of \cite{BCHolFunction}. Computing determinant from these complex representations would give a complex determinant.
\end{rem}

Next, we recall the tensor product representation of algebras.

\subsection{Tensor product representations}

\subsubsection{Tensor product of linear operators}
Let $T_1: V_1 \to W_1$ and $T_2: V_2 \to W_2$ be two linear operators. The tensor product of $T_1$ and $T_2$ is the linear map $T_1\otimes T_2 : V_1\otimes V_2 \to W_1\otimes W_2$ such that\footnote{Note that a general element of $V_1\otimes V_2$ is of the form $\sum_{i,j}{a_{ij}({v_i}\otimes{v_j}})$. Since the map $T_1\otimes T_2$ is linear, defining its action on `pure tensors' $v_i\otimes v_j$ suffices.} 
\begin{equation*}
    (T_1\otimes T_2)(v_1\otimes v_2) = (T_1(v_1))\otimes(T_2(v_2)).
\end{equation*}
for all $v_1 \in V_1$ and $v_2 \in V_2$.

If we fix basis for $V_1, W_1, V_2, W_2$, we can associate a matrix to $T_1\otimes T_2$. Let $\mathrm{A}$, $\mathrm{B}$ be the matrices corresponding to $T_1$ and $T_2$ respectively. Then the matrix corresponding to $T_1\otimes T_2$ is the Kronecker product of $\mathrm{A}$ and $\mathrm{B}$, i.e., $\mathrm{A}\otimes \mathrm{B}$. It is the block matrix\footnote{Note that sometimes the Kronecker product of two matrices $A$ and $B$ is defined to be the block matrix with blocks $a_{ij}B$ instead. This does not affect the results presented in this paper because one can reverse the ordering of the generators $\{e_1, e_2, \hdots, e_n\}$ to $\{e_n, e_{n-1}, \hdots, e_1\}$ which does not change the generated algebra $\dl{p}{q}$ but the Kronecker products, as defined in this paper, in the tensor product representation will then match with the Kronecker products defined in the alternate way.}:
\begin{equation*}
    \mathrm{A} \otimes \mathrm{B} :=
    \begin{bmatrix}
    b_{11}\mathrm{A} & b_{12}\mathrm{A} & \cdots & b_{1\delta}\mathrm{A}\\
    b_{21}\mathrm{A} & b_{22}\mathrm{A} & \cdots & b_{2\delta}\mathrm{A}\\
    \vdots & & & \vdots \\
    b_{\gamma1}\mathrm{A} & b_{\gamma2}\mathrm{A} & \cdots & b_{\gamma\delta}\mathrm{A}\\
\end{bmatrix}.
\end{equation*}

\subsubsection{Tensor product representation of algebras}

Let $A_1$ and $A_2$ be two associative algebras. Let $\gamma_1: A_1 \to \End{V_1}$ and $\gamma_2: A_2 \to \End{V_2}$ be their representations. Using $\gamma_1$ and $\gamma_2$, we can define a representation for $A_1\otimes A_2$ which we denote by $\gamma_1\otimes\gamma_2$ as follows: 
\begin{equation*}
    \gamma_1\otimes\gamma_2: A_1 \otimes A_2 \to \End{V_1\otimes V_2} \text{ , }(\gamma_1\otimes \gamma_2)(a_1\otimes a_2) = \gamma_1(a_1)\otimes\gamma_2(a_2).
\end{equation*}

\subsection{Representation of $\dl{p}{q}$}\label{Representatton of DL(p,q)}
We have described representations $\beta_+$ of $\dl{1}{0}$ and $\beta_-$ of $\dl{0}{1}$ in Section \ref{faithful reps of Dl(1,0) and DL(0,1)}. Since 
\begin{equation*}
    \dl{p}{q} \cong \underbrace{\dl{1}{0} \otimes \dl{1}{0} \cdots \otimes \dl{1}{0}}_{p \; \text{times}} \otimes \underbrace{\dl{0}{1} \otimes \dl{0}{1} \cdots \otimes \dl{0}{1}}_{q \; \text{times}},
\end{equation*}
the tensor product representation of `$p$' $\dl{1}{0}$s and `$q$' $\dl{0}{1}$s give us a representation of $\dl{p}{q}$.

Let $n := p + q$, $N := 2^n$. We define a representation $\beta_{n}: \dl{p}{q} \to \Mat(N,\R)$ as
\begin{equation}\label{The representation of K_{p,q}}
\beta_{n} := \underbrace{\beta_{+} \otimes \beta_{+} \otimes \cdots \otimes \beta_{+}}_{p \text{ times}} \otimes \underbrace{\beta_{-} \otimes \beta_{-} \otimes \cdots \otimes \beta_{-}}_{q \text{ times}}.
\end{equation}
The dimension of this representation is $N$.

Because $\beta_n$ is a tensor product representation, if $U \in \dl{p_1}{q_1}$, $V \in \dl{p_2}{q_2}$ then 
\begin{equation}\label{Representation property 2}
    \beta_{p_1+q_1+p_2+q_2}(U\otimes V) = \beta_{p_1+q_1}(U)\otimes \beta_{p_1+q_1}(V).
\end{equation}

\smallskip

Now, we express everything in terms of matrices: recall the isomorphism between $\dl{p}{q}$ and $\underbrace{\dl{1}{0} \otimes \dl{1}{0} \cdots \otimes \dl{1}{0}}_{p \; \text{times}} \otimes \underbrace{\dl{0}{1} \otimes \dl{0}{1} \cdots \otimes \dl{0}{1}}_{q \; \text{times}}$ described in Section 3.4 of \cite{CommAnalCliffAlg}:
$$1 \mapsto 1\otimes 1 \otimes \cdots \otimes 1,$$
$$e_1 \mapsto {e_1}\otimes 1 \otimes \cdots \otimes 1,$$
$$e_2 \mapsto 1\otimes {e_1} \otimes \cdots \otimes 1,$$
$$\vdots$$
$$e_j \mapsto \underbrace{1\otimes 1 \otimes \cdots \otimes 1}_{j-1 \; \text{times}} \otimes {e_1} \otimes 1 \cdots \otimes 1,$$
$$\vdots$$
$$e_n \mapsto 1\otimes 1 \otimes \cdots \otimes {e_1}.$$

Let $\mathrm{I} := \begin{bmatrix}
    1 & 0\\
    0 & 1\\
\end{bmatrix}$, $\mathrm{J}_{+}$ := $\begin{bmatrix}
    0 & 1\\
    1 & 0\\
\end{bmatrix}$ and $\mathrm{J}_{-}$ := $\begin{bmatrix}
    0 & -1\\
    1 & 0\\
\end{bmatrix}$. Using the above isomorphism, the following is $\beta_n$, the representation for $\dl{p}{q}$ expressed in terms of matrices:
$$\beta_{n}(1) \quad = \quad \mathrm{I}\otimes \mathrm{I} \otimes \cdots \otimes \mathrm{I},$$
$$\beta_{n}(e_1) \quad = \quad \mathrm{J}_\pm\otimes \mathrm{I} \otimes \cdots \otimes \mathrm{I},$$
$$\beta_{n}(e_2) \quad = \quad \mathrm{I} \otimes \mathrm{J}_\pm \otimes \cdots \otimes \mathrm{I},$$
$$\vdots$$
$$\beta_{n}(e_j) \quad = \quad \underbrace{\mathrm{I} \otimes \mathrm{I} \otimes \cdots \otimes \mathrm{I}}_{j-1 \; \text{times}} \otimes \mathrm{J}_\pm \otimes \mathrm{I} \cdots \otimes \mathrm{I},$$
$$\vdots$$
$$\beta_{n}(e_n) \quad = \quad \mathrm{I}\otimes \mathrm{I} \otimes \cdots \otimes \mathrm{J}_\pm,$$
where each $e_j$ corresponds to the Kronecker product of `$n$' $I$s except with $\mathrm{J}_+$ at $j^{th}$ spot if $e_j^2 = +1$ or $\mathrm{J}_-$ at $j^{th}$ spot if $e_j^2 = -1$. 

A nice thing about the representation $\beta_n$ is that it is faithful.

\begin{thm}
    The representation $\beta_n$ is faithful.
\end{thm}

\begin{proof}
    We give a proof using induction on $n := p + q$.
    
    Base case for induction is n = 1. We have shown in Section \ref{faithful reps of Dl(1,0) and DL(0,1)} that the matrix representation $\beta_1$ is faithful.
    
    The induction hypothesis is that the matrix representation $\beta_n$ for $\dl{p}{q}$ is faithful for $n = m$. Assuming this, we will prove that the matrix representation $\beta_n$ is faithful for $n = m+1$.
    
    Let $p + q = m + 1$. By \eqref{Tensor decomposition Property 1}, $\dl{p}{q} \cong \dl{p^\prime}{q^\prime} \otimes \dl{1}{0} \text{ or } \dl{p^\prime}{q^\prime} \otimes \dl{0}{1}$, according to if $e_{m+1}^2 = +1 \text{ or } -1$, where $p^\prime + q^\prime = m$. Let $U, V \in \dl{p}{q}$. Then because of the above isomorphism, we can express $U, V$ as $U = {x_1} + {y_1}{e_{m+1}}$, $V = {x_2} + {y_2}{e_{m+1}}$ where ${x_1}, {x_2}, {y_1}, {y_2} \in \dl{p^\prime}{q^\prime}$. Now, if $\beta_{m+1}(U) = \beta_{m+1}(V)$
    \begin{align*}
        &\implies \beta_{m+1}({x_1} + {y_1}{e_{m+1}}) = \beta_{m+1}({x_2} + {y_2}{e_{m+1}})&\\
        &\implies \beta_{m+1}({x_1}1) + \beta_{m+1}({y_1}{e_{m+1}}) = \beta_{m+1}({x_2}1) + \beta_{m+1}({y_2}{e_{m+1}})&
    \end{align*}
    By \eqref{Representation property 2}, we have that $\beta_{m+1}(\alpha\otimes \delta) = \beta_{m}({\alpha})\otimes\beta_{1}(\delta)$ for $\alpha \in \{x_1, y_1, x_2, y_2\}$ and $\delta \in \{1,e_{m+1}\}$. Thus, the above equation reduces to:
    \begin{align*}
        &\beta_{m}(x_1)\otimes \mathrm{I} + \beta_{m}(y_1)\otimes \mathrm{J}_\pm = \beta_{m}(x_2)\otimes \mathrm{I} + \beta_{m}(y_2)\otimes \mathrm{J}_\pm&\\
        &\implies 
        \begin{bmatrix}
            \beta_{m}(x_1) & \pm \beta_{m}(y_1)\\
            \beta_{m}(y_1) & \beta_{m}(x_1)
        \end{bmatrix} 
        =
        \begin{bmatrix}
            \beta_{m}(x_2) & \pm \beta_{m}(y_2)\\
            \beta_{m}(y_2) & \beta_{m}(x_2)
        \end{bmatrix}&\\
        &\implies \beta_{m}(x_1) = \beta_{m}(x_2) \text{ and } \beta_{m}(y_1) = \beta_{m}(y_2).&
    \end{align*}
    Since by our induction hypothesis, $\beta_{m}$ is faithful, we get $x_1 = x_2$ and $y_1 = y_2$ and thus $U = V$. This proves that $\beta_{m+1}$ is faithful.
\end{proof}

\subsection{Examples}
We work out the matrix representation of $\dl{2}{0}$. As $\dl{1}{0} \otimes \dl{1}{0} \cong \dl{2}{0}$, we use the matrix representation for $\dl{1}{0}$ to construct the matrix representation of $\dl{2}{0}$ by taking Kronecker product of representation matrices of $\dl{1}{0}$ as described above.

The representation is the map
$$1 \mapsto \rm{I}\otimes\rm{I} = \begin{bmatrix}
    \rm{I} & 0\\
    0 & \rm{I}\\
\end{bmatrix},$$
$${e_1} \mapsto \rm{J_+}\otimes\rm{I} = \begin{bmatrix}
    \rm{J_+} & 0\\
    0 & \rm{J_+}\\
\end{bmatrix},$$
$${e_2} \mapsto \rm{I}\otimes\rm{J_+} = \begin{bmatrix}
    0 & \rm{I}\\
    \rm{I} & 0\\
\end{bmatrix},$$
$${e_{12}} \mapsto \rm{J_+}\otimes\rm{J_+} = \begin{bmatrix}
    0 & \rm{J_+}\\
    \rm{J_+} & 0\\
\end{bmatrix},$$
i.e.,
\begin{eqnarray*}
&&a + {a_1}{e_1} + {a_2}{e_2} + {a_{12}}{e_{12}}\\
&&\mapsto \begin{bmatrix}
    a\rm{I} + {a_1}\rm{J_+} & {a_2}\rm{I} + {a_{12}}\rm{J_+}\\
    {a_2}\rm{I} + {a_{12}}\rm{J_+} & a\rm{I} + {a_1}\rm{J_+}\\
\end{bmatrix}
=
\begin{bmatrix}
    a&a_1&a_2&a_{12}\\
    a_1&a&a_{12}&a_2\\
    a_2&a_{12}&a&a_1\\
    a_{12}&a_2&a_1&a\\
\end{bmatrix}.
\end{eqnarray*}

Let us also consider the matrix representation for bicomplex numbers $\dl{0}{2}$. Similar to the matrix representation of $\dl{2}{0}$, the matrix representation of $\dl{0}{2}$ is given by 
$$1 \mapsto \rm{I}\otimes\rm{I} = \begin{bmatrix}
    \rm{I} & 0\\
    0 & \rm{I}\\
\end{bmatrix},$$
$${e_1} \mapsto \rm{J_-}\otimes\rm{I} = \begin{bmatrix}
    \rm{J_-} & 0\\
    0 & \rm{J_-}\\
\end{bmatrix},$$
$${e_2} \mapsto \rm{I}\otimes\rm{J_-} = \begin{bmatrix}
    0 & \rm{-I}\\
    \rm{I} & 0\\
\end{bmatrix},$$
$${e_{12}} \mapsto \rm{J_-}\otimes\rm{J_-} = \begin{bmatrix}
    0 & \rm{-J_-}\\
    \rm{J_-} & 0\\
\end{bmatrix},$$
i.e.,
\begin{eqnarray*}
&&a + {a_1}{e_1} + {a_2}{e_2} + {a_{12}}{e_{12}}\\
&&\mapsto \begin{bmatrix}
    a\rm{I} + {a_1}\rm{J_-} & {-a_2}\rm{I} + {-a_{12}}\rm{J_-}\\
    {a_2}\rm{I} + {a_{12}}\rm{J_-} & a\rm{I} + {a_1}\rm{J_-}\\
\end{bmatrix}
=
\begin{bmatrix}
    a&-a_1&-a_2&a_{12}\\
    a_1&a&-a_{12}&-a_2\\
    a_2&-a_{12}&a&-a_1\\
    a_{12}&a_2&a_1&a\\
\end{bmatrix}.
\end{eqnarray*}

\smallskip
One can check that the above matrix corresponding to $a + {a_1}{e_1} + {a_2}{e_2} + {a_{12}}{e_{12}}$ is precisely the Cauchy--Riemann matrix described by Price. The reason for this coincidence is that Price has defined Cauchy--Riemann matrices as the regular representation of multicomplex spaces. The way we have defined the matrix representation $\beta_n$, it is the tensor product representation of the regular representations $\beta_+$ of $\dl{1}{0}$ and $\beta_-$ of $\dl{0}{1}$. Since tensor product representation of regular representations is regular representation of tensor product algebra, we get the representation $\beta_n$ is the regular representation for $\dl{p}{q}$. In particular, for the case of multicomplex spaces $\dl{0}{n}$, we see that $\beta_n$ is the regular representation of multicomplex spaces and therefore coincides with Cauchy--Riemann matrices.

In the following sections, we will suppress the subscript $n$ in $\beta_n$ and just write $\beta$ when using the above representation.

\section{Notion of Determinant, Trace, Characteristic Polynomial}\label{Section: Notions of determinant, trace, characteristic polynomial}
Using the faithful representation $\beta$ we have defined in Section \ref{Representatton of DL(p,q)}, we can associate a `determinant' and a `trace' with elements of $\dl{p}{q}$ just like we associate a determinant and a trace with multivectors in $\cl{p}{q}$ as done in \cite{OnComputing}. Let $\beta : \dl{p}{q} \rightarrow{} \Mat(2^{p+q},\R)$ the faithful representation associated with $\dl{p}{q}$ defined in the previous section. We define {\it the determinant} as the map:
$${\rm Det}: \dl{p}{q} \rightarrow{} \R;\; \Det{U} := \det{\beta(U)},$$
and {\it the trace} as the map:
$${\rm Tr}: \dl{p}{q} \rightarrow{} \R;\; \Tr{(U)} := \tr{(\beta(U))}.$$
We also associate a characteristic polynomial with elements of $\dl{p}{q}$. Let $U \in \dl{p}{q}$. Then $\Det{U-\lambda 1}$ is a polynomial in $\lambda$ and is called {\it the characteristic polynomial associated with $U \in \dl{p}{q}$}. Since the representation $\beta$ has dimension $N = 2^{p+q}$, the characteristic polynomial associated with an element $U$ is a polynomial of degree $N$. Let $\psi_U(\lambda) = {\sum_{k=0}^{N}}{{c_k(U)}{\lambda^k}}$ be this characteristic polynomial. We call $c_k(U)$s the {\it coefficients of characteristic polynomial associated with $U$}.

Right now, the trace and determinant of an element of $\dl{p}{q}$ have to be computed by definition using the matrix representation. In the next section, we will present formulas for trace and determinant which bypass the matrix representation $\beta$ completely. 

\section{Explicit formulas for Determinant, Trace and Characteristic polynomial coefficients}\label{Explicit formulas}
In this section, we give explicit formulas for trace, determinant and other characteristic polynomial coefficients using the operations of conjugations defined in Section \ref{subsubsection: operations of conjugation}. We start with formulas for trace.

\subsection{Formulas for trace}\label{formula for trace section}
We give two formulas for trace. They are presented as the following theorems.
\begin{thm}[Formula for trace]\label{thm: formula for trace}
    An immediate consequence of the the matrix representation $\beta$ is that 
    \begin{equation}\label{formula for trace}
        \Tr{(U)} = N\proj{U}{0}
    \end{equation}
where $U \in \dl{p}{q}$ and $N:= 2^{p+q}$.
\end{thm}

This is because the only basis vector having non-zero trace is $1$. All other basis vectors have trace $0$. 
\begin{proof}
    First we have a result from matrix theory. Let $\mathrm{A}$ and $\mathrm{B}$ be two matrices. Then $\tr(A\otimes B) = \tr(A)\cdot\tr(B)$. In general, $\tr(A_1\otimes A_2\otimes \cdots \otimes A_n) = \tr(A_1)\cdot\tr(A_2)\cdots\tr(A_n)$.
    
    Since the basis vectors $e_{i_1 i_2 \cdots i_k}$ maps to the Kronecker product of $n$ $\mathrm{I}$s except with $\mathrm{J}_\pm$ at the $i_1, i_2, \hdots ,i_k$ positions, and $\tr(\mathrm{J}_\pm) = 0$, it follows that $\Tr(e_{i_1 i_2 \cdots i_k}) = 0$.
    
    Let
    \begin{multline*}
        U = u + \sum_{1 \le i \le n} u_ie_i + \!\!\sum_{1 \le i_1 < i_2 \le n}\!\! u_{i_1 i_2}e_{{i_1} {i_2}} + \quad\cdots\quad + \!\!\!\!\!\!\!\!\!\!\\ \sum_{1 \le i_i < i_2 < \cdots < i_{n-1} \le n}\!\!\!\!\!\!\!\!\!\! u_{i_1 i_2 \cdots i_{n-1}}e_{i_1 i_2 \cdots i_{n-1}} + u_{12\cdots n}e_{12\cdots n} \in \dl{p}{q}.
    \end{multline*}
    Since trace is a linear operator on matrices, $\Tr$ is linear operator on $\dl{p}{q}$. Therefore,
    \begin{align*}
        \Tr{(U)} &= \Tr\left(u + \sum_{1 \le i \le n} u_ie_i + \sum_{1 \le i_1 < i_2 \le n} u_{i_1 i_2}e_{{i_1} {i_2}} + \cdots +\right.\\
        &\left.\qquad \sum_{1 \le i_i < i_2 < \cdots < i_{n-1} \le n} u_{i_1 i_2 \cdots i_{n-1}}e_{i_1 i_2 \cdots i_{n-1}} + u_{12\cdots n}e_{12\cdots n}\right)\\
        &= \Tr(u)
        = u\cdot\Tr(1)
        = u\cdot N
        = \proj{U}{0}\cdot N. \qedhere
    \end{align*}
\end{proof}

Now we present another formula for trace which follows from the definition of operations of conjugation and simple combinatorics.
\begin{thm}[Another formula for Trace]\label{thm: another formula for trace}
    Trace of an element $\elt{p}{q}{U}$ is given by sum of all of its conjugates i.e.,
    \begin{equation}\label{formula for trace 2}
        \Tr{(U)} = \sum_{A\subseteq \firstn{n}} {\conj{U}{A}}.
    \end{equation}
\end{thm}

\begin{proof}
    Let
    \begin{multline*}
        U = u + \sum_{1 \le i \le n} u_ie_i + \!\!\sum_{1 \le i_1 < i_2 \le n}\!\! u_{i_1 i_2}e_{{i_1} {i_2}} + \quad\cdots\quad + \!\!\!\!\!\!\!\!\!\!\\ \sum_{1 \le i_i < i_2 < \cdots < i_{n-1} \le n}\!\!\!\!\!\!\!\!\!\! u_{i_1 i_2 \cdots i_{n-1}}e_{i_1 i_2 \cdots i_{n-1}} + u_{12\cdots n}e_{12\cdots n} \in \dl{p}{q}.
    \end{multline*}
    Let $A = \{i_1,i_2,\hdots, i_k\}\subseteq\firstn{n}$. Then $e_A$ is a basis element of $\dl{p}{q}$. Let $B\subseteq\firstn{n}$. Then $\conj{e_A}{B} = e_A$ if $A\cap B$ has an even number of elements and $\conj{e_A}{B} = -e_A$ if $A\cap B$ has an odd number of elements. Now, if we fix $A$, there are exactly $2^{N-1}$ subsets $B$ such that $A\cap B$ has an even number of elements and there are exactly $2^{N-1}$ subsets $B$ such that $A\cap B$ has an odd number of elements. Therefore, in the sum $\sum_{A\subseteq \firstn{n}} {\conj{U}{A}}$, all the basis elements $e_A$ cancel out. The only terms that survive in the sum is $u$ from each conjugate $\conj{U}{A}$. Therefore,
    \begin{align*}
        \sum_{A\subseteq \firstn{n}} {\conj{U}{A}} = Nu
        = N\proj{U}{0}
        =\Tr{(U)}.\tag*{\qedhere}
    \end{align*}
\end{proof}

\subsection{Formula for determinant}
We have defined the determinant of an element of $\dl{p}{q}$ using the matrix representation $\beta$. Every time we want to compute the determinant of an element $U \in \dl{p}{q}$, we would need to compute first $\beta(U)$ and then its determinant which is cumbersome. We want to find a formula which directly gives us the determinant of an element without computing its matrix representation. Such explicit formula for determinant exists, we present it in subsequent sections.

\subsubsection{Inspiration from examples with small $n$}\label{determinant in small n DL(p,q)}
Consider $\dl{0}{1}$ -- the complex numbers. Using the faithful representation defined in Section \ref{Representatton of DL(p,q)}, we compute the determinant. Let $U = a + {a_1}{e_1} \in \dl{0}{1}$. Then $\beta(U) = \begin{bmatrix}
    a & -a_1\\
    a_1 & a
\end{bmatrix}$. Finally, $\det{\beta(U)} = a^2 + {a_1}^2 = \Det{U}$. It is not hard to see that this is the good old modulus of a complex number which can also be expressed as $U\conj{U}{1}$. Therefore, $\Det{U} = U\conj{U}{1}$.

Consider $\dl{1}{0}$ -- the split-complex numbers. Using the faithful representation defined in Section \ref{Representatton of DL(p,q)}, we compute the determinant. Let $U = a + {a_1}{e_1} \in \dl{1}{0}$. Then $\beta(U) = \begin{bmatrix}
    a & a_1\\
    a_1 & a
\end{bmatrix}$. Finally, $\det{\beta(U)} = a^2 - {a_1}^2 = \Det{U}$. It is not hard to see that this is the norm of a split-complex number which can also be expressed as $U\conj{U}{1}$. Therefore, $\Det{U} = U\conj{U}{1}$.

Similarly, let us also consider $\dl{1}{1}$, the commutative quaternions. Using the faithful representation defined in Section \ref{Representatton of DL(p,q)}, we compute the determinant. Let $U = a + {a_1}{e_1} + {a_2}{e_2} + {a_{12}}{e_{12}} \in \dl{1}{1}$. Then 
$$\beta(U) = \begin{bmatrix}
    a & a_1 & -a_2 & -a_{12} \\
    a_1 & a & -a_{12} & -a_2 \\
    a_2 & a_{12} & a & a_1 \\
     a_{12} & a_2 & a_1 & a \\
\end{bmatrix}.$$ Finally, $\Det{U} = \det{\beta(U)} = ((a+{a_1})^2+({a_2} + a_{12})^2)((a-{a_1})^2+({a_2} - a_{12})^2)$ $= U\conj{U}{1}\conj{U}{2}\conj{U}{1)(2}$. The right-hand side of this expression has already been used to determine the norm of commutative quaternions in \cite{CommQuatMat, Catoni1}.

We were able to express the determinant of elements of $\dl{1}{0}$, $\dl{0}{1}$ and $\dl{1}{1}$ explicitly, without involving their matrix representations and bypassing them completely, only using operations of conjugation. This tells us that maybe we will be able to do the same in the general setting. We will see that this is indeed true and the general formula for the determinant will use the general operations of conjugations defined for $\dl{p}{q}$ in Section \ref{subsubsection: operations of conjugation}.

\subsubsection{The formula for determinant}\label{Formula for determinant section}
\begin{thm}\label{thm formula for determinant}
    Let $p,q \in \Z_{\ge0}, p+q = n$. Let $U \in \dl{p}{q}$. Then
    \begin{equation}\label{formula for determinant}
        \Det{U} = U\conj{U}{1}\conj{U}{2}\cdots\conj{U}{1)(2}\cdots\conj{U}{1)(2)(3)\cdots(n}.
    \end{equation}
\end{thm} 
In the above formula, the right-hand side is the product of all conjugates of $U$. In total, there are $N = 2^n$ factors in right-hand side corresponding to total number of different superpositions of $n$ operations of conjugation defined in Section \ref{subsubsection: operations of conjugation}. Using the notation introduced in \eqref{notation for superposition of operations of conjugation}, we can express the formula for determinant as
\begin{equation}\label{formula for determinant 2}
    \Det{U} = {\prod_{A \subseteq \firstn{n}}}{\conj{U}{A}}.
\end{equation}

The idea behind the proof is inspired from the examples considered previously: eliminate all non-zero grades by multiplication with conjugates. Let $p,q \in \Z_{\ge0}, p+q = n$. Let $\mathrm{Det}_{n}(U)$ denote the determinant of $U \in \dl{p}{q}$ given by formula \eqref{formula for determinant}. The $n$ in the subscript of $\mathrm{Det}_n$ denotes that $n$ generators are getting negated using operations of conjugation. We will prove that $\mathrm{Det}_{n}(U) = \mathrm{Det}_{n-1}(U\conj{U}{n})$, here $\mathrm{Det}_{n-1}(.)$ denotes the formula for determinant in $\dl{p^\prime}{q^\prime}$ with $p^\prime + q^\prime = n-1$. This expression makes sense because by Lemma \ref{properties of operations of conjugations 5}, $U\conj{U}{n}\in\dl{p^\prime}{q^\prime}$ Note that this is equivalent to the formula \eqref{formula for determinant}. In context of multicomplex spaces, Price mention this as a conjecture that Theorem 46.19 and Theorem 46.22 in his book hold in general for all multicomplex spaces. 

\begin{proof}[Proof of Theorem \ref{thm formula for determinant}]
    It follows from what we did in Section \ref{determinant in small n DL(p,q)} that the formula for determinant in $\C\otimes\dl{p}{q}$, where $p+q=1$, is given by the expression in \eqref{formula for determinant}. Our induction hypothesis is that formula \eqref{formula for determinant} holds for $n = m$. Assuming this, we will show that the formula then holds for $n = m+1$.
    
    Using \eqref{Tensor decomposition Property 1}, $\dl{p}{q} \cong \dl{p^\prime}{q^\prime} \otimes \dl{1}{0} \text{ or } \dl{p^\prime}{q^\prime} \otimes \dl{0}{1}$, according to if $e_{m+1}^2 = +1 \text{ or } -1$, where $p^\prime + q^\prime = m$. Let $U \in \dl{p}{q}$. Then because of the above isomorphism, we can express $U, V$ as $U = x + y{e_{m+1}}$, where $x, y \in \dl{p^\prime}{q^\prime}$. Using \eqref{Representation property 2},
    \begin{align*}
        \beta_{m+1}(U) &= \beta_{m+1}(x + y{e_{m+1}})\\
        &= \beta_{m+1}(x1) + \beta_{m+1}(ye_{m+1})\\
        &= \beta_m(x)\otimes\beta_1(1) + \beta_m(y)\otimes\beta_1(e_{m+1}) = \beta_m(x)\otimes \mathrm{I} + \beta_m(y)\otimes\mathrm{J}_\pm\\ 
        &= \begin{bmatrix}
            \beta_m(x) & \pm\beta_m(y)\\
            \beta_m(y) & \beta_m(x)
        \end{bmatrix}.
    \end{align*}  
    \smallskip
    When $A,B,C,D$ commute, $\mathrm{det}\left(\begin{bmatrix}
        A & B\\
        C & D
    \end{bmatrix}\right) = \mathrm{det}(AD-BC)$ (check out \cite{DetBlockMatrix} for details). Since $\beta_m(x), \beta_m(y)$ commute, we get 
    \begin{align*}
        \mathrm{Det}_{m+1}(U) = \mathrm{det}\left(\begin{bmatrix}
        \beta_m(x) & \pm\beta_m(y)\\
        \beta_m(y) & \beta_m(x)
        \end{bmatrix}\right)
        = \mathrm{det}(\beta_m(x)^2 \mp \beta_m(y)^2)\\
        = \mathrm{det}(\beta_m(x^2\mp y^2))
        = \mathrm{Det}(x^2\mp y^2)
        = \mathrm{Det}(U\conj{U}{m+1})
    \end{align*}

    Therefore, we see that $\mathrm{Det}_{m+1}(U) = {\mathrm{Det}_m}(U\conj{U}{m+1})$ holds which is equivalent to formula \eqref{formula for determinant} as discussed above. This completes the proof.
\end{proof} 

\subsubsection{Interesting observations about the determinant}
The formula \eqref{formula for determinant} is very elegant and interesting. The first interesting thing is that the formula  for determinant is independent of signature $(p,q)$. Also, it follows that $U$ and all its conjugates have the same determinant.
\begin{thm}\label{all conjugates of U have same determinant theorem}
    Let $\elt{p}{q}{U}$, then
    \begin{equation}\label{all conjugates of U have same determinant}
        \Det{\conj{U}{A}} = \Det{U} \text{ for all } A \subseteq \firstn{n}.
    \end{equation}
\end{thm}
\begin{proof}
    To prove this, we use the formula \eqref{formula for determinant 2} for determinant:
     \begin{equation*}
         \Det{U} = {\prod_{A \subseteq \firstn{n}}}{\conj{U}{A}}.
     \end{equation*}
     Recall from remark \ref{operations of conjugations form an abelian group} that the set of all operations of conjugation forms a group under superposition of operations of conjugation. Once we make this observation, the result becomes obvious because 
     \begin{equation}\label{proof that all conjugates of U have same determinant}
         \Det{\conj{U}{B}} = {\prod_{A \subseteq \firstn{n}}}{\conj{U}{B)(A}}.
     \end{equation}
     and since $(B)(A)$ is a permutation of elements $(A)$ and the product \eqref{proof that all conjugates of U have same determinant} has the same terms as in the product \eqref{formula for determinant 2}.
\end{proof}

\subsection{Formulas for coefficients of characteristic polynomials}
Recall that the characteristic polynomial associated with $U \in \dl{p}{q}$ is the polynomial in $\lambda$, $\psi_U(\lambda) = \Det{U-\lambda 1}$. Now that we have the formula for determinant \eqref{formula for determinant}, we can use it to find the formulas for characteristic polynomial coefficients. First, we present some interesting observations about characteristic polynomials. 

Let $p,q\in\Z_{\ge0}, n:=p+q$. Let $\elt{p}{q}{U}$. Let $\psi_U(\lambda) = \Det{U-\lambda1}$ denote the characteristic polynomial of U.
\begin{thm}\label{characteristic polynomial of all conjugates is same}
    Let $\elt{p}{q}{U}$, $A\subseteq\firstn{n}$. Then $\psi_{\conj{U}{A}}(\lambda) = \psi_U(\lambda)$.
\end{thm}
\begin{proof}
    From Theorem \ref{all conjugates of U have same determinant theorem}, we know that all conjugates of $U$ have the same determinant, implying that they have the same characteristic polynomial as well. Explicitly, let $A\subseteq\firstn{n}$, then $\Det{U-\lambda1} = \Det{\conj{(U-\lambda1)}{A}} = \Det{\conj{U}{A}-\lambda1}$ implying that $\psi_{\conj{U}{A}}(\lambda) = \psi_U(\lambda)$ for all $A\subseteq\firstn{n}$.
\end{proof}

\begin{cor}
    Let $\elt{p}{q}{U}$. Then $\conj{U}{A}$ satisfies the characteristic equation of $U$ i.e., $\psi_U(\conj{U}{A}) = 0$ for all $A\subseteq\firstn{n}$.
\end{cor}
\begin{proof}
    By the Cayley--Hamilton Theorem, $U$ satisfies its characteristic equation i.e., $\psi_U{(U)} = 0$. Since from the above Theorem \ref{characteristic polynomial of all conjugates is same}, all conjugates of $U$ have the same characteristic polynomial i.e., $\psi_{\conj{U}{A}}(\lambda) = \psi_U(\lambda)$, it implies that $\psi_U(\conj{U}{A}) = 0$ for all $A\subseteq\firstn{n}$.
\end{proof}

\subsubsection{Explicit formulas for coefficients of characteristic polynomials}
Let $\psi_U(\lambda) = {\sum_{k=0}^{N}}{{c_k(U)}{\lambda^k}}$ be the characteristic polynomial where $c_k(U)$s are the coefficients of the characteristic polynomial as defined in Section \ref{Section: Notions of determinant, trace, characteristic polynomial}. We present explicit formulas for them as the following theorem.

\begin{thm}\label{thm formula for characteristic polynomial coefficients}
Let $\psi_U(\lambda) = {\sum_{k=0}^{N}}{{c_k(U)}{\lambda^k}}$ be the characteristic polynomial of $U \in \dl{p}{q}$. Then
\begin{equation}\label{formula for characteristic polynomial coefficients}
        \begin{aligned}
            &c_N(U) = (-1)^N = 1,\\
            &c_k(U) = (-1)^{k}\sum_{A_1,A_2,\hdots,A_{N-k} \subseteq \firstn{n}}\left(\prod_{1\le i \le N-k}\conj{U}{A_i}\right) \text{ for } 0\le k<N,
        \end{aligned}
    \end{equation}
    where the sum is over all distinct subsets of $\firstn{n}$ $N-k$ at a time. 
\end{thm}
\smallskip
\begin{proof}
     One can explicitly express the characteristic polynomial of $U$, $\psi_U$

    \begin{multline*}
        \psi_U(\lambda) = {\sum_{k=0}^{N}}{{c_k(U)}{\lambda^k}}
        = \Det{U-\lambda 1} = (U-\lambda 1)\conj{(U-\lambda 1)}{1}\conj{(U-\lambda 1)}{2}\\
        \cdots\conj{(U-\lambda 1)}{1)(2}\cdots\conj{(U-\lambda 1)}{1)(2)(3)\cdots(n}.
    \end{multline*}
    
    Now, the result follows from comparing coefficients of $\lambda^k$ on both sides of the above equation.
\end{proof}

Notice that $c_0(U)$ gives us the formula for determinant of $U$ that we know from Section \ref{Formula for determinant section} and $-c_{N-1}(U)$ is the trace of $U$ and thus gives us the formula for trace of $U$ which we had got in Theorem \ref{thm: another formula for trace}.

\begin{rem} 
    Note that we could get the same formulas for coefficients of characteristic polynomial \eqref{formula for characteristic polynomial coefficients} by using Vieta's formulas\footnote{Compare with the more non-trivial non-commutative Vieta formulas for the case of Clifford algebras presented in the paper \cite{OnNoncomm}.}. Since all conjugates of $U$, $\{\conj{U}{A} \;|\; A\subseteq\firstn{n}\}$ satisfy the characteristic polynomial of $U$, $\psi_U$ one can argue using Vieta's Theorem that the coefficients of characteristic polynomial $\psi_U$ are symmetric polynomials of $\{\conj{U}{A} \;|\; A\subseteq\firstn{n}\}$ i.e., the coefficients of the characteristic polynomial do not change under a permutation of the elements in the set $\{\conj{U}{A} \;|\; A\subseteq\firstn{n}\}$ which are precisely the expressions \eqref{formula for characteristic polynomial coefficients}.
\end{rem}

\section{Multiplicative inverses in $\dl{p}{q}$}\label{Section: mult inverses}
Unlike the problem of finding multiplicative inverses in Clifford algebra, the problem of finding multiplicative inverses in commutative analogues of Clifford algebra is easier because we have explicit formulas for determinant in general. Since a matrix is invertible if and only if its determinant is non-zero, it immediately follows that $U \in \dl{p}{q}$ is invertible if and only if $\Det{U} \ne 0$.

Define ${\rm Adj}:\dl{p}{q} \to \dl{p}{q}$,
\begin{eqnarray}
{\rm Adj}(U) := \conj{U}{1}\conj{U}{2}\cdots\conj{U}{1)(2}\cdots\conj{U}{1)(2)(3)\cdots(n},
\end{eqnarray}
such that $\Det{U}$ $ = U {\rm Adj}(U)$. We call ${\rm Adj}(U)$ \textit{the adjoint} of the element $U$.

\subsection{Explicit formulas for multiplicative inverses}
\begin{thm}\label{formula for inverse theorem}
    If $\Det{U} \ne 0$, then the inverse of $U$ is given by 
    $$\frac{{\rm Adj}(U)}{\Det{U}}.$$
\end{thm}
\begin{proof}
    The proof follows from that fact that $\Det{U} = U{\rm Adj}(U)$. Therefore, if $\Det{U}\ne0$, then $U^{-1} = \frac{{\rm Adj}(U)}{\Det{U}}.$
\end{proof}
We give explicit formulas for ${\rm Adj}(U)$ for $n \le 4$ below:
$${\rm Adj}(U) = \left\lbrace
\begin{array}{ll}
      \conj{U}{1}, & \mbox{if n = 1};\\
      \conj{U}{1}\conj{U}{2}\conj{U}{12}, & \mbox{if n = 2};\\
      \conj{U}{1}\conj{U}{2}\conj{U}{3}\conj{U}{12}\conj{U}{23}\conj{U}{13}\conj{U}{123}, & \mbox{if n = 3};\\
      \conj{U}{1}\conj{U}{2}\conj{U}{3}\conj{U}{4}\conj{U}{12}\conj{U}{13}\conj{U}{14}\conj{U}{23}\\
      \,\,\cdot\,\conj{U}{24}\conj{U}{34}\conj{U}{123}\conj{U}{134}\conj{U}{124}\conj{U}{234}\conj{U}{1234}, & \mbox{if n = 4}.
\end{array}
\right.
$$
In context of multicomplex spaces, the above formulas for $n=2,3$ are known as mentioned earlier. The formula for a general multicomplex space, according to the best of our knowledge, is not known. 

\begin{thm}[Inverse of conjugate] 
    It is the case that $\left(\conj{U}{A}\right)^{-1} = \conj{\left(U^{-1}\right)}{A}.$
\end{thm}
\begin{proof}
    This follows from ${\rm Adj}(\conj{U}{A}) = \conj{{\rm Adj}(U)}{A}$. 
\end{proof}
\subsection{A classification of elements in $\dl{p}{q}$ based on $\mathrm{Det}$}\label{Section: A classification of elements in DL(p,q) based on Det}
An element $U \in \dl{p}{q}$ is invertible if and only if $\Det{U} \ne 0$. What happens when $\Det{U} = 0$? This is answered by the following theorem.

\begin{thm}\label{thm: non-invertible elements are zero divisors}
$\elt{p}{q}{U}$ is a zero divisor (i.e., there exists a non-zero element $V \in \dl{p}{q}$ such that $UV=0$) if and only if $\Det{U}=0$.
\end{thm}
\begin{proof}
In one direction the statement is obvious. If $U$ is a zero divisor, then there exists $V\ne0$ such that $UV=0$. Suppose that $\Det{U}\neq 0$, i.e. $U$ is invertible. Then, multiplying $UV=0$ on the left by $U^{-1}$, we get $V=0$, i.e. a contradiction.

        For the other direction, let  
        $\Det{U} = 0$. The idea is to start eliminating terms involving $e_i$s one by one by multiplying an element with its conjugates keeping in mind Lemma \ref{properties of operations of conjugations 5}. Consider the products $U\conj{U}{j}$, $j \in \firstn{n}$. If for some $j$, $U\conj{U}{j} = 0$, then we are done because we can take $V = \conj{U}{j}$. If not, consider the products $W\conj{W}{j}$, $j \in \{2,3,\hdots,n\}$ where $W = U\conj{U}{1}$. If for some $j$, $W\conj{W}{j} = 0$, then we are done because we can take $V = \conj{U}{1}\conj{U}{j}\conj{U}{1)(j}$. We continue repeating this process. The worst case is that the only product that will be zero is $\Det{U}$, but then we are again done.
\end{proof}
This means that the non-invertible elements in $\dl{p}{q}$ are zero divisors. The same result holds for Clifford algebras as well and follows because Clifford algebras are isomorphic to matrix algebras.

\begin{rem}
Note that we can also calculate all coefficients of the characteristic polynomial, adjoint, and inverse of $U\in\dl{p}{q}$ using the well-known Faddeev–LeVerrier algorithm (a recursive method to calculate the coefficients of the characteristic polynomial from the matrix theory)\footnote{Similarly to how we do it in (noncommutative) Clifford algebras, see \cite{OnComputing}.}:
\begin{eqnarray}
U_{(1)}:=U,\quad U_{(k+1)}:=U(U_{(k)}-c_k),\quad c_k=\frac{N}{k}\langle U_{(k)} \rangle_0,\quad 1\le k \le N,
\end{eqnarray}
in particular,
\begin{eqnarray}
\Det{U}=-U_{(N)}=-c_N=U(c_{N-1}-U_{(N-1)}).
\end{eqnarray}
If $\Det{U}\neq 0$, then
\begin{eqnarray}
U^{-1}=\frac{{\rm Adj}(U)}{\Det{U}},\qquad {\rm Adj}(U)=c_{N-1}-U_{(N-1)}.
\end{eqnarray}
\end{rem}

\section{Conclusions}\label{section_conclusions}
In this paper, we completely solve the problem of multiplicative inverses in commutative analogues of Clifford algebras $\dl{p}{q}$. In order to do that, we give a matrix representation for $\dl{p}{q}$ and using the representation, associate the notion determinant. We also associate notions of trace and characteristic polynomial coefficients. It is to be noted that commutativity allowed for explicit formulas, i.e., formulas that bypass the matrix representation completely, for trace (Theorems \ref{formula for trace} and \ref{thm: another formula for trace}), formula for determinant (Theorem \ref{thm formula for determinant}), formula for characteristic polynomials coefficients (Theorem \ref{thm formula for characteristic polynomial coefficients}) and formula for multiplicative inverse (Theorem \ref{formula for inverse theorem}). 

The presented explicit formulas for determinant in $\dl{p}{q}$ \eqref{formula for determinant} and \eqref{formula for determinant 2} are a single product of conjugates. This is in contrast with the formulas for determinant in Clifford algebras $\cl{p}{q}$ where explicit formulas for determinant are in general a linear combination of terms involving conjugates (we have an algorithm to calculate it, check Theorem 4 in \cite{OnComputing} and Theorem 3 in \cite{OnNoncomm}). For $p+q\le5$, formula for determinant is a single product (check \cite{OnComputing}), for $p+q=6$ the formula is a linear combination of 2 terms (check \cite{A.Acus}) and for $p+q\ge7$ the number of terms in the formula for determinant is not known. Our main aim for this paper was to address the multiplicative inverse problem in $\dl{p}{q}$ i.e. how to find multiplicative inverses in $\dl{p}{q}$ because we believe that solving the multiplicative inverse problem in commutative analogues might help us better understand the multiplicative inverse problem in $\cl{p}{q}$ and allow us to find more explicit formulas for determinants for $p+q\ge7$ with least number of terms in the expression for determinant. The connection between commutative analogues of Clifford algebras and Clifford algebras is made by some subalgebras of $\cl{p}{q}$ being isomorphic to $\dl{p^\prime}{q^\prime}$ that cause trouble in Clifford algebras by forcing the formula for determinant to be a linear combination rather than a single product (check out \cite{A.Acus}). This is a topic to be further explored.

As commutative analogues of Clifford algebras generalise multicomplex spaces, the formulas for multiplicative inverses (Theorem \ref{formula for inverse theorem}) given in this paper would also work for multicomplex spaces. The formula given in this paper generalises the known formulas for multiplcative inverses in bicomplex and tricomplex spaces \cite{PlatonicSolids, DynamicsTricomplex,Price, BCHolFunction, Shapiro}.

Also, we expect the use of commutative analogues of Clifford algebras $\dl{p}{q}$ in various applications, just like bicomplex spaces and commutative quaternions. Bicomplex spaces are used in the study of Mandelbrot set \cite{SlicesOfMandelbrodSet,FrenchMSThesis,MSThesis,PlatonicSolids}, quantum mechanics \cite{BCQM1,BCQM2,MCSchrodingerEQn}, and in many other applications \cite{FinDImBCHilbertSpace,BCRiemannZetaFunc,MCRiemannZetaFunc,BCPolygammaFunction}. Commutative quaternions are widely used in computer science \cite{ZhangWangVasilevSVD, ZhangJiangVasilevLeastSQ, ZhangGuoWangJiangLeastSQ, ZhangWangVasilevJiangLeastSQ, Isokawa, Kobayashi} and theoretical physics \cite{Catoni3, ZhangWangVasilevJiangMaxwellEqn}. In particular, since Theorem \ref{formula for inverse theorem} gives multiplicative inverses in the algebras $\dl{p}{q}$, we believe it to be useful in all applications that require solving linear equations with coefficients in $\dl{p}{q}$.

\section*{Acknowledgment}
The first author deeply thanks Prof. Dmitry Shirokov for giving him the opportunity to work with him. The first author also wishes to thank his friends, especially Chitvan Singh and Ipsa Bezbarua, who were always there for him.

The authors are grateful to the anonymous reviewers for their careful reading of the paper and helpful comments on how to improve the presentation.

The article was prepared within the framework of the project “Mirror Laboratories” HSE University “Quaternions, geometric algebras and applications”.
\medskip

{\bf Data availability} Data sharing is not applicable to this article as no datasets were generated or analyzed during the current study.

\medskip

\noindent{\bf Declarations}\\
\noindent{\bf Conflict of interest} The authors declare that they have no conflict of interest.


\bibliographystyle{spmpsci}

\end{document}